\documentclass[12pt,oneside,reqno]{amsart}

\usepackage{amsmath,amssymb,amsthm,textcomp}
\usepackage{amsfonts,graphicx}
\usepackage[mathscr]{eucal}
\usepackage{color}
\usepackage{csquotes}
\usepackage[backend=bibtex,%
firstinits=true,%
doi=false,%
isbn=true,%
url=false,%
maxnames=99]{biblatex}%

\AtEveryBibitem{\clearfield{issn}}
\AtEveryCitekey{\clearfield{issn}}
\numberwithin{equation}{section}
\DeclareNameAlias{sortname}{last-first}
\theoremstyle{definition}
\usepackage{mathtools}
\numberwithin{equation}{section}

\newcommand{\ncom}{\newcommand}

\ncom{\beq}{\begin{equation}}
\ncom{\eeq}{\end{equation}}
\ncom{\bea}{\begin{eqnarray*}}
\ncom{\eea}{\end{eqnarray*}}
\ncom{\beqa}{\begin{eqnarray}}
\ncom{\eeqa}{\end{eqnarray}}
\ncom{\nno}{\nonumber}
\ncom{\non}{\nonumber}
\ncom{\ds}{\displaystyle}
\ncom{\half}{\frac{1}{2}}
\ncom{\mbx}{\makebox{.25cm}}
\ncom{\hs}{\mbox{\hspace{.25cm}}}
\ncom{\rar}{\rightarrow}
\ncom{\Rar}{\Rightarrow}
\ncom{\noin}{\noindent}
\ncom{\bc}{\begin{center}}
\ncom{\ec}{\end{center}}
\ncom{\sz}{\scriptsize}
\ncom{\rf}{\ref}
\ncom{\s}{\sqrt{2}}
\ncom{\sgm}{\sigma}
\ncom{\Sgm}{\Sigma}
\ncom{\psgm}{\sigma^{\prime}}
\ncom{\dt}{\delta}
\ncom{\Dt}{\Delta}
\ncom{\lmd}{\lambda}
\ncom{\Lmd}{\Lambda}
\ncom{\Th}{\Theta}
\ncom{\e}{\eta}
\ncom{\eps}{\epsilon}
\ncom{\pcc}{\stackrel{P}{>}}
\ncom{\lp}{\stackrel{L_{p}}{>}}
\ncom{\dist}{{\rm\,dist}}
\ncom{\sspan}{{\rm\,span}}
\ncom{\re}{{\rm Re\,}}
\ncom{\im}{{\rm Im\,}}
\ncom{\sgn}{{\rm sgn\,}}
\ncom{\ba}{\begin{array}}
\ncom{\ea}{\end{array}}
\ncom{\hone}{\mbox{\hspace{1em}}}
\ncom{\htwo}{\mbox{\hspace{2em}}}
\ncom{\hthree}{\mbox{\hspace{3em}}}
\ncom{\hfour}{\mbox{\hspace{4em}}}
\ncom{\vone}{\vskip 2ex}
\ncom{\vtwo}{\vskip 4ex}
\ncom{\vonee}{\vskip 1.5ex}
\ncom{\vthree}{\vskip 6ex}
\ncom{\vfour}{\vspace*{8ex}}
\ncom{\norm}{\|\;\;\|}
\ncom{\integ}[4]{\int_{#1}^{#2}\,{#3}\,d{#4}}
\ncom{\vspan}[1]{{{\rm\,span}\{ #1 \}}}
\ncom{\dm}[1]{ {\displaystyle{#1} } }
\ncom{\ri}[1]{{#1} \index{#1}}

\newtheorem{theorem}{\bf Theorem}[section]
\newtheorem{remark}{\bf Remark}[section]
\newtheorem{proposition}{Proposition}[section]
\newtheorem{lemma}{Lemma}[section]
\newtheorem{corollary}{Corollary}[section]

\newtheorem{definition}{Definition}[section]
\newtheoremstyle
    {remarkstyle}
    {}
    {11pt}
    {}
    {}
    {\bfseries}
    {:}
    {     }
    {\thmname{#1} \thmnumber{#2} }

\theoremstyle{remarkstyle}

\def\eps{\varepsilon}

\begin{document}
\title{Generalized Fractional Counting Process}
\author[Kuldeep Kumar Kataria]{Kuldeep Kumar Kataria}
\address{Kuldeep Kumar Kataria, Department of Mathematics, Indian Institute of Technology Bhilai, Raipur 492015, India.}
\email{kuldeepk@iitbhilai.ac.in}
\author[Mostafizar Khandakar]{Mostafizar Khandakar}
\address{Mostafizar Khandakar, Department of Mathematics, Indian Institute of Technology Bhilai, Raipur 492015, India.}
\email{mostafizark@iitbhilai.ac.in}
\date{June 22, 2021}
\begin{abstract}
 In this paper, we obtain additional results for a fractional counting process introduced and studied by Di Crescenzo {\it et al.} (2016). For convenience, we call it the generalized fractional counting process (GFCP). It is shown that the one-dimensional distributions of the GFCP are not infinitely divisible. Its covariance structure is studied using which its long-range dependence property is established. It is shown that the increments of GFCP exhibits the  short-range dependence property. Also, we prove that the GFCP is a scaling limit of some continuous time random walk. A particular case of the GFCP, namely, the generalized counting process (GCP) is discussed for which we obtain  a limiting result, a martingale result and establish a recurrence relation for its probability mass function. We have shown that many known counting processes such as the Poisson process of order $k$, the P\'olya-Aeppli process of order $k$, the negative binomial process and their fractional versions {\it etc.} are other special cases of the GFCP. An application of the GCP to risk theory is discussed. 
\end{abstract}

\maketitle
\section{Introduction}
The time fractional Poisson process (TFPP) is a renewal process with heavy-tailed distributed waiting times (see Laskin (2003), Beghin and Orsingher (2009)). Biard and  Saussereau (2014) showed that its increment process exhibits the long-range dependence (LRD) property and discussed its applications in risk theory. The processes that exhibit the LRD property has applications in several areas, for example, finance (see Ding {\it et al.} (1993)), hydrology (see Doukhan {\it et al.} (2003), pp. 461-472), internet data traffic modeling (see Karagiannis {\it et al.} (2004)) {\it etc.} For other fractional versions of the Poisson process we refer the reader to Beghin (2012), Orsingher and Polito (2012), Kataria and Vellaisamy (2017) {\it etc.} and references therein.

Let $\{M^\alpha(t)\}_{t\geq0}$, $0<\alpha\leq 1$ be a fractional counting process which performs $k$ kinds of jumps of amplitude $1,2,\dots,k$ with positive rates $\lambda_{1},\lambda_{2},\ldots,\lambda_{k}$, respectively, where $k\geq1$ is a fixed integer and whose state probabilities $p^\alpha(n,t)=\mathrm{Pr}\{M^\alpha(t)=n\}$ satisfy
\begin{equation}\label{cre}
\frac{\mathrm{d}^{\alpha}}{\mathrm{d}t^{\alpha}}p^\alpha(n,t)=-\Lambda p^\alpha(n,t)+	\sum_{j=1}^{\min\{n,k\}}\lambda_{j}p^\alpha(n-j,t),\ \ n\ge0,
\end{equation}
with
\begin{equation*}
p^\alpha(n,0)=\begin{cases}
1,\ \ n=0,\\
0,\ \ n\ge1.
\end{cases}
\end{equation*}
Here, $\Lambda=\lambda_{1}+\lambda_{2}+\dots+\lambda_{k}$ and $\dfrac{\mathrm{d}^{\alpha}}{\mathrm{d}t^{\alpha}}$ is the Caputo fractional derivative which is defined as
\begin{equation*}
\frac{\mathrm{d}^{\alpha}}{\mathrm{d}t^{\alpha}}f(t)\coloneqq\begin{cases}
\dfrac{1}{\Gamma \left( 1-\alpha \right)}\displaystyle\int_{0}^{t}(t-s)^{-\alpha}f^{\prime}(s)\,\mathrm{d}s,\ \ 0<\alpha<1,\vspace*{.2cm}\\
f^{\prime}(t), \ \ \alpha=1.
\end{cases}
\end{equation*}

The process $\{M^\alpha(t)\}_{t\geq0}$ is introduced and studied by Di Crescenzo {\it et al.} (2016). Throughout this paper, we call it the generalized fractional counting process (GFCP). Its probability mass function (pmf) is given by (see Di Crescenzo {\it et al.} (2016))
\begin{equation}\label{pmfmt}
p^{\alpha}(n,t)
=\sum_{r=0}^{n}\, \sum_{\substack{i_{1}+i_{2}+\ldots+i_{k}=r\\ i_{1}+2i_{2}+\ldots+ki_{k} =n}}\binom{r}{i_{1},i_{2},\ldots,i_{k}}\, \lambda _{1}^{i_{1}}\lambda _{2}^{i_{2}}\ldots\lambda _{k}^{i_{k}}\, t^{r\alpha}E_{\alpha,r\alpha +1 }^{r+1 }(-\Lambda t^{\alpha}),\ \ n\ge0,
\end{equation}
where $i_1,i_2,\ldots,i_k$ are non-negative integers and $E_{\alpha,r\alpha +1 }^{r+1 }(\cdot)$ is the three-parameter Mittag-Leffler function defined in (\ref{mit}). Its mean and variance are given by 
\begin{equation}\label{mean-variance}
\mathbb{E}\left(M^{\alpha}(t)\right)=St^{\alpha},\ \ \operatorname{ Var}\left(M^{\alpha}(t)\right)=Rt^{2\alpha}+Tt^{\alpha},
\end{equation}
where 
\begin{equation*}
S=\frac{\sum_{j=1}^{k}j\lambda_{j}}{\Gamma(\alpha+1)},\ \ R=\left(\frac{2}{\Gamma(2\alpha+1)}-\frac{1}{\Gamma^2(\alpha+1)}\right)\left(\sum_{j=1}^{k}j\lambda_{j}\right)^{2},\  T=\frac{\sum_{j=1}^{k}j^{2}\lambda_{j}}{\Gamma(\alpha+1)}.
\end{equation*}

For $k=1$, the GFCP reduces to TFPP. A limiting case of the GFCP, namely, the convoluted fractional Poisson process (CFPP) which is obtained by taking suitable $\lambda_{j}$'s and letting $k\to\infty$ is studied by Kataria and Khandakar (2021). For $\alpha=1$, the GFCP reduces to a special case, namely, the generalized counting process (GCP). We discuss the GCP in detail later in the paper.

Di Crescenzo {\it et al.} (2016) showed that
\begin{equation}\label{compound}
M^\alpha(t)\overset{d}{=}\sum_{i=1}^{N^\alpha(t)}X_{i},\ \ t\ge0,
\end{equation}
where $\overset{d}{=}$ denotes equal in distribution. Here, $\{N^\alpha(t)\}_{t\ge0}$ is the TFPP with intensity parameter $\Lambda$ which is independent of the sequence of independent and identically distributed (iid) random variables $\{X_{i}\}_{i\ge1}$ such that
\begin{equation}\label{prxi=j}
\mathrm{Pr}\{X_{1}=j\}=\frac{\lambda_{j}}{\Lambda},\ \ j=1,2,\dots,k.
\end{equation}

It is also known that
\begin{equation}\label{mya}
M^\alpha(t)\stackrel{d}{=}M(Y_{\alpha}(t)),
\end{equation}
where the inverse $\alpha$-stable subordinator $\{Y_{\alpha}(t)\}_{t\ge0}$ is independent of the GCP $\{M(t)\}_{t\geq0}$.

In this paper, we study some additional results for the GFCP and for its  special case, the GCP. In Section \ref{section2}, some known results on the Mittag-Leffler function and the inverse $\alpha$-stable subordinator are provided. In Section \ref{section3}, we obtain the  characteristic function and the L\'evy measure of GCP. It is shown that the process $\{M(t)-\sum_{j=1}^{k}j\lambda_{j}t\}_{t\geq0}$ is a martingale with respect to a suitable filtration. The following recurrence relation for the pmf $p(n,t)=\mathrm{Pr}\{M(t)=n\}$ of GCP is obtained:
\begin{equation*}
p(n,t)=\frac{t}{n}\sum_{j=1}^{\min\{n,k\}}j\lambda_{j}p(n-j,t),  \ \ n\ge1.	
\end{equation*}
Also, we have shown that
\begin{equation*}
\lim_{t\to\infty}\frac{M(t)}{t}=\sum_{j=1}^{k}j\lambda_{j},\ \ \text{in probability}.
\end{equation*}	
The above limiting result is used to show that the one-dimensional distributions of GFCP are not infinitely divisible. The explicit expressions for the probability generating function (pgf) and the $r$th factorial moment of GFCP are obtained. Its LRD property is established by utilizing its covariance. Also, it is shown that the increments of GFCP has the short-range dependence (SRD) property. We discuss a continuous time random walk (CTRW) whose scaling limit is the GFCP. 

In Section \ref{section4}, it is shown that some known counting processes such as the Poisson process of order $k$, the P\'olya-Aeppli process of order $k$, the negative binomial process and their fractional versions {\it etc.} are special cases of the GFCP. Some results for these particular as well as limiting cases are obtained. 

In Section \ref{section5}, we considered a risk model in which the GCP is used to model the number of claims received. The governing differential equation for the joint probability of the time to ruin and the deficit at the time of ruin is derived for the introduced risk model. The closed form expression for its ruin probability with no initial capital is obtained.

\section{Preliminaries}\label{section2}
Here, we provide some known results related to Mittag-Leffler function and
 inverse $\alpha$-stable subordinator. These results will be required later.
\subsection{Mittag-Leffler function}
The three-parameter Mittag-Leffler function is defined as
\begin{equation}\label{mit}
	E_{\beta, \gamma}^{\delta}(x)\coloneqq\frac{1}{\Gamma(\delta)}\sum_{k=0}^{\infty} \frac{\Gamma(\delta+k)x^{k}}{k!\Gamma(k\beta+\gamma)},\ \ x\in\mathbb{R},
\end{equation}
where $\beta>0$, $\gamma>0$ and $\delta>0$. It reduces to the two-parameter and the one-parameter Mittag-Leffler function for $\delta=1$ and $\delta=\gamma=1$, respectively. The following holds true (see Kilbas {\it et al.} (2006), Eq. (1.8.22)):
\begin{equation}\label{re}
	E_{\beta, \gamma}^{(n)}(x)=n!E_{\beta, n\beta+\gamma}^{n+1}(x),\ \ n\ge0.
\end{equation}
where $E_{\beta, \gamma}^{(n)}(\cdot)$ denotes the $n$th derivative of two-parameter Mittag-Leffler function.
\subsection{Inverse $\alpha$-stable subordinator}
A $\alpha$-stable subordinator $\{D_{\alpha}(t)\}_{t\ge0}$, $0<\alpha<1$ is a non-decreasing L\'evy process. Its Laplace transform is given by $\mathbb{E}\left(e^{-sD_{\alpha}(t)}\right)=e^{-ts^{\alpha}}$, $s>0$. Its first passage time $\{Y_{\alpha}(t)\}_{t\ge0}$ is called the inverse $\alpha$-stable subordinator and it is defined as 
\begin{equation*}
Y_{\alpha}(t)=\inf\{x>0: D_{\alpha}(x)> t\}.
\end{equation*}

The mean of $Y_{\alpha}(t)$ is given by (see Leonenko {\it et al.} (2014))
\begin{equation}\label{xswswa3}
\mathbb{E}\left(Y_{\alpha}(t)\right)=\frac{t^{\alpha}}{\Gamma(\alpha+1)}.
\end{equation}

Let $B(\alpha,\alpha+1)$ and $B(\alpha,\alpha+1;s/t)$ denote the beta function and  the incomplete beta function, respectively. It is known that (see Leonenko {\it et al.} (2014), Eq. (10))
\begin{equation}\label{covvdsw}
\operatorname{Cov}\left(Y_{\alpha}(s),Y_{\alpha}(t)\right)=\frac{1}{\Gamma^2(\alpha+1)}\left( \alpha s^{2\alpha}B(\alpha,\alpha+1)+F(\alpha;s,t)\right),
\end{equation}
where $0<s\leq t$ and $F(\alpha;s,t)=\alpha t^{2\alpha}B(\alpha,\alpha+1;s/t)-(ts)^{\alpha}$. On using the following asymptotic result (see Maheshwari and  Vellaisamy (2016), Eq. (8)):
\begin{equation*}
F(\alpha;s,t)\sim \frac{-\alpha^{2}}{(\alpha+1)}\frac{s^{\alpha+1}}{t^{1-\alpha}},\ \ \mathrm{as}\ \ t\to\infty,
\end{equation*}
in (\ref{covvdsw}), we get the following result for fixed $s$ and large $t$:
\begin{equation}\label{asi}
\operatorname{Cov}\left(Y_{\alpha}(s),Y_{\alpha}(t)\right)\sim \frac{1}{\Gamma^{2}(\alpha+1)}\left(\alpha s^{2\alpha} B(\alpha,\alpha+1)-\frac{\alpha^{2}}{(\alpha+1)}\frac{s^{\alpha+1}}{t^{1-\alpha}}\right).
\end{equation}

\section{Generalized fractional counting process}\label{section3}
In this section, we obtain some additional results for the GFCP and its special case, the GCP. 

Di Crescenzo {\it et al.} (2016) showed that the GFCP $\{M^{\alpha}(t)\}_{t\geq0}$ is not a L\'evy process. However, it can be seen from (\ref{compound}) that for the case $\alpha=1$, {\it i.e.}, the GCP $\{M(t)\}_{t\geq0}$ is equal in distribution to a compound Poisson process which is a L\'evy process. Thus, the GCP is a L\'evy process and its characteristic function is given by
\begin{align}\label{ch}
\mathbb{E}\left(e^{i\xi M(t)}\right)&=\mathbb{E}\left(e^{i\xi\sum_{i=1}^{N(t)}X_{i}}\right),\ \ \xi\in\mathbb{R}\nonumber\\
&=\exp\left(-\mathbb{E}\left(N(t)\right)\left(1-\mathbb{E}\left(e^{i\xi X_{1}}\right)\right)\right)\nonumber\\
&=\exp\left(-t\sum_{j=1}^{k}(1-e^{i\xi j})\lambda_{j}\right),
\end{align}
where the last step follows from (\ref{prxi=j}). Also, its L\'evy measure is given by
\begin{equation}\label{qap675}
\Pi(\mathrm{d}x)=\sum_{j=1}^{k}\lambda_{j}\delta_{j}\mathrm{d}x,
\end{equation}
where $\delta_{j}$'s are Dirac measures.

The pgf $G^{\alpha}(u,t)=\mathbb{E}\left(u^{M^{\alpha}(t)}\right)$ of GFCP is given by
\begin{equation}\label{pgf}
G^{\alpha}(u,t)=E_{\alpha,1}\left(\sum_{j=1}^{k}\lambda_{j}(u^{j}-1)t^{\alpha}\right),\ \ |u|\le1,
\end{equation}
whose proof follows similar lines to that of Proposition 2.1, Di Crescenzo {\it et al.} (2016).

Next, we obtain a recurrence relation for the state probabilities of GCP.
\begin{proposition}
	The  state probabilities $p(n,t)$ of GCP satisfy
	\begin{equation*}
p(n,t)=\frac{t}{n}\sum_{j=1}^{\min\{n,k\}}j\lambda_{j}p(n-j,t),  \ \ n\ge1.	
	\end{equation*}
\end{proposition}
\begin{proof}
	From the definition of pgf, we get
	\begin{equation*}
	\frac{\mathrm{d}}{\mathrm{d}u}G(u,t)=\sum_{i=0}^{\infty}(i+1)p(i+1,t)u^{i}.
	\end{equation*}
On substituting $\alpha=1$ in (\ref{pgf}) and taking derivative, we get 
\begin{equation*}
\frac{\mathrm{d}}{\mathrm{d}u}G(u,t)=t\sum_{j=1}^{k}j\lambda_{j}u^{j-1}G(u,t).
\end{equation*}
On equating the above two equations, we get
\begin{align*}
\sum_{i=0}^{\infty}(i+1)p(i+1,t)u^{i}&=t\sum_{j=1}^{k}j\lambda_{j}u^{j-1}\sum_{i=0}^{\infty}p(i,t)u^i\\
&=t\sum_{j=1}^{k}j\lambda_{j}\sum_{i=j-1}^{\infty}p(i-j+1,t)u^{i}\\
&=t\sum_{j=1}^{k}j\lambda_{j}\left(\sum_{i=j-1}^{k-2}p(i-j+1,t)u^{i}+\sum_{i=k-1}^{\infty}p(i-j+1,t)u^{i}\right)\\
&=t\sum_{i=0}^{k-2}\sum_{j=1}^{i+1}j\lambda_{j}p(i-j+1,t)u^{i}+t\sum_{i=k-1}^{\infty}\sum_{j=1}^{k}j\lambda_{j}p(i-j+1,t)u^{i}.
\end{align*}
On equating the coefficients of $u^i$ for $0\le i\le k-2$, we get
\begin{equation*}
(i+1)p(i+1,t)=t\sum_{j=1}^{i+1}j\lambda_{j}p(i-j+1,t)
\end{equation*}
which reduces to
\begin{equation}\label{nlek}
p(n,t)=\frac{t}{n}\sum_{j=1}^{n}j\lambda_{j}p(n-j,t),\ \ 1\le n\le k-1.
\end{equation}
Again on equating the coefficients of $u^i$ for $i\ge k-1$, we get
\begin{equation*}
(i+1)p(i+1,t)=t\sum_{j=1}^{k}j\lambda_{j}p(i-j+1,t)
\end{equation*}
which reduces to 
\begin{equation}\label{ngek}
p(n,t)=\frac{t}{n}\sum_{j=1}^{k}j\lambda_{j}p(n-j,t),\ \ n\ge k.
\end{equation}
Finally, the result follows on combining (\ref{nlek}) and (\ref{ngek}).
\end{proof}
\begin{proposition}\label{martingale}
	The process $\{M(t)-\sum_{j=1}^{k}j\lambda_{j}t\}_{t\geq0}$ is a martingale with respect to natural filtration $\mathscr{F}_{t}=\sigma\left(M(s), s\le t\right)$.
\end{proposition}
\begin{proof}
	Let $Q(t)=M(t)-\sum_{j=1}^{k}j\lambda_{j}t$. Note that $M(t)$ has independent increments as it's a L\'evy process. Hence, for $s\le t$, we have 
	\begin{equation*}
	\mathbb{E}\left(Q(t)-Q(s)|\mathscr{F}_{s}\right)=	\mathbb{E}\left(M(t)-M(s)\Big|\mathscr{F}_{s}\right)-\sum_{j=1}^{k}j\lambda_{j}(t-s)=0.
	\end{equation*}
	This completes the proof.
\end{proof}
\begin{lemma}\label{lemma}
The following limiting result holds for GCP:
\begin{equation}\label{limit}
\lim_{t\to\infty}\frac{M(t)}{t}=\sum_{j=1}^{k}j\lambda_{j},\ \ \text{in probability}.
\end{equation}	
\end{lemma}
\begin{proof}
On substituting $\alpha=1$ in (\ref{pgf}), we get the pgf of GCP as 
\begin{equation}\label{pgfm}
G(u,t)=\prod_{j=1}^{k}e^{t\lambda_{j}(u^{j}-1)}.
\end{equation}
Thus, the GCP is equal in distribution to the following weighted sum of $k$ independent Poisson process:
\begin{equation*}
M(t)\stackrel{d}{=}\sum_{j=1}^{k}jN_{j}(t).
\end{equation*}	
The weighted Poisson process $N_{1}(t)+2N_{2}(t)+\cdots+kN_{k}(t)$ is studied by Zuo {\it et al.} (2021). Here, $\{N_{j}(t)\}_{t\ge0}$ is a Poisson process with intensity $\lambda_{j}$. Thus, 
\begin{align*}
\lim_{t\to\infty}\frac{M(t)}{t}&\stackrel{d}{=}\sum_{j=1}^{k}j\lim_{t\to\infty}\frac{N_{j}(t)}{t}\\
&=\sum_{j=1}^{k}j\lambda_{j},\ \ \text{in probability},
\end{align*}
where we have used $\lim_{t\to\infty}N_{j}(t)/t=\lambda_j$ almost surely. This completes the proof.
\end{proof}
\begin{remark}
Kataria and Khandakar (2021) studied a limiting case of the GCP, namely, the convoluted Poisson process (CPP). It is denoted by $\{\mathcal{N}_{c}(t)\}_{t\ge0}$. Let $\{\beta_{j}\}_{j\in\mathbb{Z}}$ be a sequence of intensity parameters such that $\beta_{j}=0$ for all $j<0$ and $\beta_{j}>\beta_{j+1}>0$ for all $j\geq0$ with $\lim\limits_{j\to\infty}\beta_{j+1}/\beta_{j}<1$. On taking $\lambda_{j}=\beta_{j-1}-\beta_{j}$, $j\ge 1$ and letting $k\to \infty$ in (\ref{cre}) with $\alpha=1$, the GCP reduces to the CPP. Thus, from (\ref{limit}), the following holds for the CPP:
	\begin{equation*}
	\lim_{t\to\infty}\frac{\mathcal{N}_{c}(t)}{t}=\sum_{j=0}^{\infty}\beta_{j},\ \ \text{in probability}.
	\end{equation*}
\end{remark}
\begin{remark}
	From (\ref{pgfm}), we note that the GCP can be represented as a sum of $k$ independent compound Poisson processes $\{C_{j}(t)\}_{t\geq0}$, $j=1,2,\dots,k$ where 
	\begin{equation*}
	C_j(t)=\sum_{i=1}^{N_j(t)}X_{i}.
	\end{equation*}
Here, $X_i=j$ with probability $1$ and $\{N_j(t)\}_{t\ge0}$ is the Possion process with intensity $\lambda_{j}$.
\end{remark}
\begin{proposition}
The one-dimensional distributions of  GFCP are not infinitely divisible.
\end{proposition}
\begin{proof}
On using the self-similarity property of inverse $\alpha$-stable subordinator $\{Y_{\alpha}(t)\}_{t\ge0}$ in (\ref{mya}), we get
\begin{equation*}
M^{\alpha}(t)\overset{d}{=}M\left(t^{\alpha}Y_{\alpha}(1)\right).
\end{equation*}
Thus,
\begin{align*}
\lim_{t\to\infty}\frac{M^{\alpha}(t)}{t^{\alpha}}&\overset{d}{=}\lim_{t\to\infty}\frac{M\left(t^{\alpha}Y_{\alpha}(1)\right)}{t^{\alpha}}\\
&=Y_{\alpha}(1)\lim_{t\to\infty}\frac{M\left(t^{\alpha}Y_{\alpha}(1)\right)}{t^{\alpha}Y_{\alpha}(1)}\\
&\overset{d}{=}Y_{\alpha}(1)\sum_{j=1}^{k}j\lambda_{j},
\end{align*}
where we have used Lemma \ref{lemma} in the last step. 
Now, let  us assume that $M^{\alpha}(t)$ is infinitely divisible. Thus, $M^{\alpha}(t)/t^{\alpha}$ is  infinitely divisible. It follows that  $Y_{\alpha}(1)$ is infinitely divisible as $\lim_{t\to\infty}M^{\alpha}(t)/t^{\alpha}$ is infinitely divisible which follows by using a result on p. 94 of Steutel and van Harn (2004). This leads to a contradiction as $Y_{\alpha}(1)$ is not infinitely divisible (see Vellaisamy and Kumar (2018)).
\end{proof}

Next we obtain the factorial moments of GFCP by using its pgf.
\begin{proposition}\label{p3.3}
The $r$th factorial moment of GFCP, that is, $\psi^\alpha(r,t)=	\mathbb{E}(M^{\alpha}(t)(M^{\alpha}(t)-1)\ldots(M^{\alpha}(t)-r+1))$, $r\ge1$, is given by
\begin{equation}\label{factorial}
\psi^\alpha(r,t)=r!\sum_{n=1}^{r}\frac{t^{n\alpha}}{\Gamma(n\alpha+1)}\underset{m_i\in\mathbb{N}}{\underset{\sum_{i=1}^nm_i=r}{\sum}}\prod_{\ell=1}^{n}\left(\frac{1}{m_\ell!}\sum_{j=1}^{k}(j)_{m_\ell}\lambda_{j}\right),
\end{equation}
where $(j)_{m_\ell}=j(j-1)\ldots(j-m_\ell+1)$ denotes the falling factorial.
\end{proposition}
\begin{proof}
On using the $r$th derivative of composition of two functions (see Johnson (2002), Eq. (3.3))	in (\ref{pgf}), we get
\begin{align}\label{ttt}
\psi^\alpha(r,t)&=\frac{\partial^{r}G^{\alpha}(u,t)}{\partial u^{r}}\bigg|_{u=1}\nonumber\\
&=\sum_{n=0}^{r}\frac{1}{n!}E^{(n)}_{\alpha,1}\left(t^{\alpha}\sum_{j=1}^{k}\lambda_{j}(u^{j}-1)\right)	\left.B_{r,n}\left(t^{\alpha}\sum_{j=1}^{k}\lambda_{j}(u^{j}-1)\right)\right|_{u=1},
\end{align}
where
\begin{align}\label{mkgtrr4543t}
B_{r,n}&\left(t^{\alpha}\sum_{j=1}^{k}\lambda_{j}(u^{j}-1)\right)\Bigg|_{u=1}\nonumber\\
&=\sum_{m=0}^{n}\frac{n!}{m!(n-m)!}\left(-t^{\alpha}\sum_{j=1}^{k}\lambda_{j}(u^{j}-1)\right)^{n-m}\frac{\mathrm{d}^{r}}{\mathrm{d}u
^{^{r}}}\left(t^{\alpha}\sum_{j=1}^{k}\lambda_{j}(u^{j}-1)\right)^{m}\Bigg|_{u=1}\nonumber\\
&=t^{n\alpha }\frac{\mathrm{d}^{r}}{\mathrm{d}u^{^{r}}}\left(\sum_{j=1}^{k}\lambda_{j}(u^{j}-1)\right)^{n}\Bigg|_{u=1}.
\end{align}
From (\ref{re}), we get
\begin{align}\label{ppp}
E^{(n)}_{\alpha,1}\left(t^{\alpha}\sum_{j=1}^{k}\lambda_{j}(u^{j}-1)\right)\Bigg|_{u=1}&=n!E^{n+1}_{\alpha,n\alpha+1}\left(t^{\alpha}\sum_{j=1}^{k}\lambda_{j}(u^{j}-1)\right)\Bigg|_{u=1}\nonumber\\
&=\frac{n!}{\Gamma(n\alpha+1)}.
\end{align}
Now, by using the following result (see Johnson (2002), Eq. (3.6))
\begin{equation*}
\frac{\mathrm{d}^{r}}{\mathrm{d}w^{^{r}}}(g(w))^{n}=\underset{m_i\in\mathbb{N}_0}{\underset{m_{1}+m_{2}+\dots+m_{n}=r}{\sum}}\frac{r!}{m_1!m_2!\ldots m_n!}g^{(m_{1})}(w)g^{(m_{2})}(w)\dots g^{(m_{n})}(w),
\end{equation*}
we get		
\begin{align}\label{ccc}
\frac{\mathrm{d}^{r}}{\mathrm{d}u^{^{r}}}\left(\sum_{j=1}^{k}\lambda_{j}(u^{j}-1)\right)^{n}\Bigg|_{u=1}&=r!\underset{m_i\in\mathbb{N}_0}{\underset{\sum_{i=1}^nm_i=r}{\sum}}\prod_{\ell=1}^{n}\frac{1}{m_\ell!}\frac{\mathrm{d}^{m_{\ell}}}{\mathrm{d}u^{{m_{\ell}}}}\left(\sum_{j=1}^{k}\lambda_{j}(u^{j}-1)\right)\Bigg|_{u=1}\nonumber\\
&=r!\underset{m_i\in\mathbb{N}}{\underset{\sum_{i=1}^nm_i=r}{\sum}}\prod_{\ell=1}^{n}\frac{1}{m_\ell!}\sum_{j=1}^{k}(j)_{m_\ell}\lambda_{j}.
\end{align}	
The right hand side of (\ref{ccc}) vanishes for $n=0$. The proof follows on substituting (\ref{mkgtrr4543t}) and (\ref{ppp}) in (\ref{ttt}) and then using (\ref{ccc}).	
\end{proof}
\begin{remark}
On substituting $k=1$  in (\ref{factorial}), we get
\begin{equation*}
\psi^\alpha(r,t)=r!\sum_{n=1}^{r}\frac{t^{n\alpha}}{\Gamma(n\alpha+1)}\underset{m_i\in\mathbb{N}}{\underset{\sum_{i=1}^nm_i=r}{\sum}}\prod_{\ell=1}^{n}\left(\frac{1}{m_\ell!}(1)_{m_\ell}\lambda_{1}\right)=\frac{r!(\lambda_{1}t^{\alpha})^{r}}{\Gamma(r\alpha+1)},
\end{equation*}
which is the $r$th factorial moment of TFPP (see Beghin and Orsingher (2009), Eq. (2.9)).
\end{remark}
\begin{remark}
Di Crescenzo {\it et al.} (2016) give an expression for the $r$th moment $\mu^\alpha(r,t)=\mathbb{E}\left((M^{\alpha}(t))^{r}\right)$, $r\ge1$ of GFCP. Here, we give an alternate expression for it as follows:
\begin{equation*}
\mu^\alpha(r,t)=r!\sum_{n=1}^{r}\frac{t^{n\alpha}}{\Gamma(n\alpha+1)}\underset{m_i\in\mathbb{N}}{\underset{\sum_{i=1}^nm_i=r}{\sum}}\prod_{\ell=1}^{n}\left(\frac{1}{m_\ell!}\sum_{j=1}^{k}j^{m_\ell}\lambda_{j}\right),
\end{equation*}
whose proof follows along the similar lines to that of  Proposition \ref{p3.3}.
\end{remark}
\subsection{Dependence structure of the GFCP and its increments}
Here, we show that the GFCP has LRD property whereas its increments exhibits the SRD property.
  
The LRD and SRD properties for a non-stationary stochastic process $\{X(t)\}_{t\geq0}$ are defined as follows (see Maheshwari and Vellaisamy (2016)):
\begin{definition}
	Let $s>0$ be fixed and $\{X(t)\}_{t\ge0}$ be a stochastic process whose correlation function satisfies
	\begin{equation*}
	\operatorname{Corr}(X(s),X(t))\sim c(s)t^{-\theta},\  \text{as}\ t\rightarrow\infty,
	\end{equation*}
for some $c(s)>0$. The process $\{X(t)\}_{t\ge0}$ is said to exhibit the LRD property if $\theta\in(0,1)$ and the SRD property if $\theta\in(1,2)$.
\end{definition}
We use Theorem 2.1 of Leonenko {\it et al.} (2014) to obtain the  covariance of GFCP as follows: Let $0<s\le t$. Then,
 \begin{align}
\operatorname{Cov}\left(M^{\alpha}(s),M^{\alpha}(t)\right)&=\operatorname{Var}\left(M(1)\right)\mathbb{E}(Y_{\alpha}(s))+\left(\mathbb{E}(M(1))\right)^{2}\operatorname{Cov}\left(Y_{\alpha}(s),Y_{\alpha}(t)\right)\nonumber\\
&=Ts^{\alpha}+\left(\sum_{j=1}^{k}j\lambda_{j}\right)^{2}\operatorname{Cov}\left(Y_{\alpha}(s),Y_{\alpha}(t)\right),\label{covfrd11}
\end{align}
where we have used $(\ref{mean-variance})$ with $\alpha=1$ and (\ref{xswswa3}) in the last step. 

Now, on using (\ref{asi}) in (\ref{covfrd11}), we obtain
\begin{equation}\label{covzt}
\operatorname{Cov}\left(M^{\alpha}(s),M^{\alpha}(t)\right)\sim Ts^{\alpha}+ S^{2}\left(\alpha s^{2\alpha}B(\alpha,\alpha+1)-\frac{\alpha^{2}}{(\alpha+1)}\frac{s^{\alpha+1}}{t^{1-\alpha}}\right) \ \ \mathrm{as}\  t\to\infty.
\end{equation}
\begin{remark}
The mean and variance of the GFCP are obtained by Di Crescenzo {\it et al.} (2016). Alternatively, these can be obtained from Theorem 2.1 of Leonenko {\it et al.} (2014).
\end{remark}
\begin{theorem}
The GFCP exhibits the LRD property.
\end{theorem}
\begin{proof}
Using (\ref{mean-variance}) and (\ref{covzt}), we get the following for fixed $s>0$ and large $t$:
\begin{align*}
\operatorname{Corr}\left(M^{\alpha}(s),M^{\alpha}(t)\right)
&\sim\frac{Ts^{\alpha}+ S^{2}\left(\alpha s^{2\alpha}B(\alpha,\alpha+1)-\frac{\alpha^{2}}{(\alpha+1)}\frac{s^{\alpha+1}}{t^{1-\alpha}}\right)}{\sqrt{\operatorname{ Var}\left(M^{\alpha}(s)\right)}\sqrt{Rt^{2\alpha}+Tt^{\alpha}}}\\
&\sim c_0(s)t^{-\alpha},
\end{align*}
where
\begin{equation*}
c_0(s)=\frac{\Gamma(2\alpha+1)Ts^{\alpha}+\left(\sum_{j=1}^{k}j\lambda_{j}\right)^{2}s^{2\alpha}}{\Gamma(2\alpha+1)\sqrt{\operatorname{ Var}\left(M^{\alpha}(s)\right)R}}.
\end{equation*}
As $0<\alpha<1$, the result follows.
\end{proof}
Similarly, it can be shown that the GCP exhibits the LRD property.

For a fixed $h>0$, the increment process of GFCP is defined as  
\begin{equation*}
Z^{\alpha}_{h}(t)\coloneqq M^{\alpha}(t+h)-M^{\alpha}(t),\ \ t\geq0.
\end{equation*}
\begin{theorem}\label{varbgfff}
The increment process $\{Z^{\alpha}_{h}(t)\}_{t\ge0}$ has the SRD property.
\end{theorem}
\begin{proof}
The proof follows similar lines to that of Theorem 1, Maheshwari and Vellaisamy (2016) and Theorem 5.5, Kataria and Khandakar (2021).  For the sake of completeness, we give a brief outline of the proof.

Let $s>0$ be fixed such that $0< s+h\le t$. Then,
\begin{align}\label{covz}
\operatorname{Cov}(Z^{\alpha}_{h}(s),Z^{\alpha}_{h}(t))&=\operatorname{Cov}\left(M^{\alpha}(s+h),M^{\alpha}(t+h)\right)+\operatorname{Cov}\left(M^{\alpha}(s),M^{\alpha}(t)\right)\nonumber\\
&\ \ \ \  -\operatorname{Cov}\left(M^{\alpha}(s+h),M^{\alpha}(t)\right)-\operatorname{Cov}\left(M^{\alpha}(s),M^{\alpha}(t+h)\right).
\end{align}
 From (\ref{covzt}) and (\ref{covz}), we get the following for large $t$:
\begin{equation}\label{covzi}
\operatorname{Cov}(Z^{\alpha}_{h}(s),Z^{\alpha}_{h}(t))\sim\frac{\alpha^{2}h(1-\alpha)}{\alpha+1}\left((s+h)^{\alpha+1}-s^{\alpha+1}\right)S^{2}t^{\alpha-2}.
\end{equation}

On using (\ref{covvdsw}) in (\ref{covfrd11}), we get
\begin{equation}\label{23}
\operatorname{Cov}\left(M^{\alpha}(t),M^{\alpha}(t+h)\right)= Tt^{\alpha}+ S^{2}\left(\alpha t^{2\alpha}B(\alpha,\alpha+1)+F(\alpha;t,t+h)\right),
\end{equation}
where $F(\alpha;t,t+h)=\alpha(t+h)^{2\alpha}B(\alpha,\alpha+1;t/(t+h))-\left(t(t+h)\right)^{\alpha}$.

Also,
\begin{align}\label{varzi}
\operatorname{Var}(Z^{\alpha}_{h}(t))&=\operatorname{Var}\left(M^{\alpha}(t+h)\right)+\operatorname{Var}\left(M^{\alpha}(t)\right)-2\operatorname{Cov}\left(M^{\alpha}(t),M^{\alpha}(t+h)\right)\nonumber\\
&\sim \alpha h Tt^{\alpha-1},\ \ \mathrm{as}\ \ t\to\infty,
\end{align}
where we have used (\ref{mean-variance}), (\ref{23}) and the result $B(\alpha,\alpha+1;t/(t+h))\sim B(\alpha,\alpha+1)$ for large $t$ in the last step. Finally, from (\ref{covzi}) and (\ref{varzi}), we get
\begin{equation*}
\operatorname{Corr}(Z^{\alpha}_{h}(s),Z^{\alpha}_{h}(t))\sim c_1(s)t^{-(3-\alpha)/2},\ \ \mathrm{as}\ t\rightarrow\infty,
\end{equation*}
where
\begin{equation*}
c_1(s)=\frac{\alpha^{2}h(1-\alpha)\left((s+h)^{\alpha+1}-s^{\alpha+1}\right)S^{2}}{(\alpha+1)\sqrt{\operatorname{Var}(Z^{\alpha}_{h}(s))}\sqrt{\alpha h T}}.
\end{equation*}
Thus, the process $\{Z^{\alpha}_{h}(t)\}_{t\ge0}$ exhibits the SRD property as $1<(3-\alpha)/2<3/2$.
\end{proof}

\subsection{GFCP as a scaling limit of a CTRW}
Consider a renewal process
\begin{equation*}
	R(t)=\max\{n\geq 0:W_1+W_2+\cdots+W_n\leq t\},
\end{equation*}
where $W_1,W_2,\ldots,W_n$ are iid waiting times such that $\mathrm{Pr}\{W_n>t\}=t^{-\alpha}L(t)$, $0<\alpha<1$ and $L$ is a slowly varying function. Then, there exist $b_n>0$ such that
\begin{equation*}
	b_n(W_1+W_2+\dots+W_n)\Rightarrow D_{\alpha}(1),
\end{equation*}
where $\Rightarrow$ denotes convergence in distribution. It means that $W_1$ belongs to the strict domain of attraction of some stable law $D_{\alpha}(1)$.  Let $b(t)=b_{[t]}$.  It can be shown that there exists a regularly varying function
$\tilde b$ with index $\alpha$ such that $1/b(\tilde b(c))\sim c$, as $c\to\infty$ (see Meerschaert {\it et al.} (2011)). 

Let $S^{(p)}(n)=\sum_{i=1}^{n}Z_{i}$ where $Z_{i}=X_{i}V^{(p)},\ i\ge1$. Here, $\{X_{i}\}_{i\ge1}$ is a sequence of iid random variables whose distribution is given by (\ref{prxi=j}) and $V^{(p)}$ is a Bernoulli random variable independent of  $\{X_{i}\}_{i\ge1}$ such that $\mathrm{Pr}\{V^{(p)}=1\}=p$ and $\mathrm{Pr}\{V^{(p)}=0\}=1-p$. Note that $S^{(p)}(R(t))$ is a CTRW with heavy-tailed
waiting times and jumps distributed according to the law of $Z_{1}$. 

The following result holds for the GFCP: 
\begin{equation}\label{CTRWtoCFPP}
\big\{S^{(1/\tilde b(c))}([\Lambda R(ct)])\big\}_{t\geq 0}\Rightarrow \left\{M^\alpha(t))\right\}_{t\geq 0}
\end{equation}
as $c\to\infty$ in the $M_1$ topology on $D\left([0,\infty),\mathbb{R}\right)$. That is, the GFCP is the scaling limit of a CTRW. The result given in (\ref{CTRWtoCFPP}) can be proved along the similar lines to that of Theorem 4.8 of Kataria and Khandakar (2021). Thus, the proof is omitted.
\section{Some special cases of the GFCP}\label{section4}
In this section, we discuss few special cases of the GFCP. It is known that the TFPP and the CFPP are particular and limiting cases of the GFCP, respectively (see Di Crescenzo {\it et al.} (2016), Kataria and Khandakar (2021)). Its other special cases are as follow:
\subsection{Poisson process of order $k$ and its fractional version}
The Poisson process of order $k$ (PPoK) $\{N^{k}(t)\}_{t\ge0}$ is a compound Poisson process introduced and studied by Kostadinova and Minkova (2013). It is defined as  
\begin{equation*}
	N^{k}(t)\coloneqq\sum_{i=1}^{N(t)}X_{i},
\end{equation*}
where $\{X_{i}\}_{i\ge1}$ is a sequence of iid discrete uniform random variables such that 
\begin{equation*}
\mathrm{Pr}\{X_{1}=j\}=\frac{1}{k}, \ \ j=1,2,\dots,k.
\end{equation*}
The sequence $\{X_{i}\}_{i\ge1}$ is independent of the Poisson process $\{N(t)\}_{t\ge0}$ whose intensity parameter is $k\lambda$. For $k=1$, the PPoK reduces to the Poisson process. Recently, a fractional version of the PPoK, namely, the time fractional Poisson process of order $k$ (TFPPoK) $\{N^{k}_{\alpha}(t)\}_{t\ge0}$ is studied by Gupta and Kumar (2021), Kadankova {\it et al.} (2021). It is defined as 
\begin{equation*}
	N^{k}_{\alpha}(t)\coloneqq N^{k}(Y_{\alpha}(t)),\ \ 0<\alpha<1,
\end{equation*}
where the PPoK $\{N^{k}(t)\}_{t\ge0}$ and the inverse $\alpha$-stable subordinator $\{Y_{\alpha}(t)\}_{t\ge0}$ are independent of each other.

On substituting $\lambda_{j}=\lambda$ for all $j=1,2,\dots,k$ in (\ref{cre}), we get the governing system of fractional differential equations for the state probabilities of TFPPoK (see Gupta and Kumar (2021), Eq. (30); Kadankova {\it et al.} (2021), Eqs. (18)-(19)). Further, on taking $\alpha=1$, we get the governing system of differential equations for the state probabilities of PPoK (see Kostadinova and Minkova (2013), Eq. (9)). Here, $\Lambda=k\lambda$. Thus, the PPoK and its fractional version TFPPoK are particular cases of the GFCP.

The pmf $p^{k}_{\alpha}(n,t)=\mathrm{Pr}\{N^{k}_{\alpha}(t)=n\}$ of TFPPoK is obtained by Gupta and Kumar (2021), Kadankova {\it et al.} (2021). Its alternate form can be obtained by substituting $\lambda_{j}=\lambda$ for $j=1,2,\dots,k$ in (\ref{pmfmt}), and it is given by
\begin{equation*}
p^{k}_{\alpha}(n,t)
=\sum_{r=0}^{n}\, \sum_{\substack{i_{1}+i_{2}+\ldots+i_{k}=r\\ i_{1}+2i_{2}+\ldots+ki_{k} =n}}\binom{r}{i_{1},i_{2},\ldots,i_{k}}\, (\lambda t^{\alpha})^{r}\, E_{\alpha,r\alpha +1 }^{r+1 }(-k\lambda t^{\alpha }),\ \ n\ge0.
\end{equation*}
Similarly, the $r$th factorial moment of TFPPoK can be obtained from Proposition \ref{p3.3}. Moreover, the characterstic function of PPoK (see Gupta {\it et al.} (2020), Eq. (10)) and a limiting result for PPoK (see Sengar {\it et al.} (2020), Eq. (9)) follow from (\ref{ch}) and (\ref{limit}), respectively.

\subsection{ P\'olya-Aeppli process of order $k$ and its fractional version}
The P\'olya-Aeppli process of order $k$ (PAPoK) $\{\hat{N}^{k}(t)\}_{t\ge0}$ is a compound Poisson process  studied by Chukova and Minkova (2015). It is defined as  
\begin{equation*}
	\hat{N}^{k}(t)\coloneqq\sum_{i=1}^{N(t)}X_{i},
\end{equation*}
where $\{X_{i}\}_{i\ge1}$ is a sequence of iid truncated geometrically distributed random variables with the following pmf: 
\begin{equation*}
	\mathrm{Pr}\{X_{1}=j\}=\frac{1-\rho}{1-\rho^{k}}\rho^{j-1}, \ \ j=1,2,\dots,k,
\end{equation*}
where $0\leq\rho<1$. The sequence $\{X_{i}\}_{i\ge1}$ is independent of the Poisson process $\{N(t)\}_{t\ge0}$ whose intensity parameter is $\lambda$. Recently, a fractional version of the PAPoK, namely, the fractional P\'olya-Aeppli process of order $k$ (FPAPoK) $\{\hat{N}^{k}_{\alpha}(t)\}_{t\ge0}$ is introduced by  Kadankova {\it et al.} (2021). It is defined as 
\begin{equation*}
	\hat{N}^{k}_{\alpha}(t)\coloneqq \hat{N}^{k}(Y_{\alpha}(t)),\ \ 0<\alpha<1,
\end{equation*}
where $\{\hat{N}^{k}(t)\}_{t\ge0}$ and $\{Y_{\alpha}(t)\}_{t\ge0}$ are independent of each other.

On substituting $\lambda_{j}=\lambda(1-\rho)\rho^{j-1}/(1-\rho^{k})$ for all $j=1,2,\dots,k$ in (\ref{cre}), we get the governing system of fractional differential equations for the state probabilities of FPAPoK (see Kadankova {\it et al.} (2021), Eqs. (37)-(38)). Further, on taking $\alpha=1$, we get the governing system of differential equations for the state probabilities of PAPoK (see Chukova and Minkova (2015
), Eq. (9)). Here, $\Lambda=\lambda$. Thus, the PAPoK and its fractional version FPAPoK are particular cases of the GFCP. From Lemma \ref{lemma}, we get the following limiting result for the PAPoK:
	\begin{equation*}
		\lim_{t\to\infty}\frac{\hat{N}^{k}(t)}{t}=\frac{\lambda}{1-\rho^{k}}\left(1+\rho+\dots+\rho^{k-1}-k\rho^{k}\right),\ \ \text{in probability}.
	\end{equation*}	
It is important to note that the pmf of FPAPoK is not known. On substituting  $\lambda_{j}=\lambda(1-\rho)\rho^{j-1}/(1-\rho^{k})$ for $j=1,2,\dots,k$ in (\ref{pmfmt}), the pmf $\hat{p}^{k}_{\alpha}(n,t)=\mathrm{Pr}\{\hat{N}^{k}_{\alpha}(t)=n\}$ of FPAPoK can be obtained as follows:
\begin{equation*}
\hat{p}^{k}_{\alpha}(n,t)
=\sum_{r=0}^{n}\, \sum_{\substack{i_{1}+i_{2}+\ldots+i_{k}=r\\ i_{1}+2i_{2}+\ldots+ki_{k} =n}}\binom{r}{i_{1},i_{2},\ldots,i_{k}}\, \left(\frac{\lambda(1-\rho)t^{\alpha}}{\rho(1-\rho^{k})}\right)^{r}\, \rho^{n}E_{\alpha,r\alpha +1 }^{r+1 }(-\lambda t^{\alpha }),\ \ n\ge0
\end{equation*}
and its pgf can be obtained from (\ref{pgf}) in the following form:
\begin{equation*}
\hat{G}^{k}_{\alpha}(u,t)=E_{\alpha,1}\left(\frac{\lambda(1-\rho)}{(1-\rho^{k})}\sum_{j=1}^{k}\rho^{j-1}(u^{j}-1)t^{\alpha}\right),\ \ |u|\le1.
\end{equation*}
Further, $\alpha=1$ gives the pmf $\hat{p}^{k}(n,t)=\mathrm{Pr}\{\hat{N}^{k}(t)=n\}$ of PAPoK as follows:
\begin{equation*}
\hat{p}^{k}(n,t)
=\sum_{r=0}^{n}\, 
\sum_{\substack{i_{1}+i_{2}+\ldots+i_{k}=r\\ i_{1}+2i_{2}+\ldots+ki_{k} =n}}\left(\frac{\lambda(1-\rho)t}{\rho(1-\rho^{k})}\right)^{r}
\frac{\rho^{n}e^{-\lambda t}}
{i_{1}!\,i_{2}!\,\ldots\,i_{k}!},\ \ n\ge0.
\end{equation*}
From (\ref{qap675}), it follows that its L\'evy measure is
\begin{equation*}
\Pi(\mathrm{d}x)=\frac{\lambda(1-\rho)}{(1-\rho^{k})}\sum_{j=1}^{k}\rho^{j-1}\delta_{j}\mathrm{d}x.
\end{equation*}

\subsection{P\'olya-Aeppli process and its fractional version}
The P\'olya-Aeppli process  (PAP) $\{\hat{N}(t)\}_{t\ge0}$ is a compound Poisson process  studied by Chukova and Minkova (2013). It is defined as  
\begin{equation*}
\hat{N}(t)\coloneqq\sum_{i=1}^{N(t)}X_{i},
\end{equation*}
where $\{X_{i}\}_{i\ge1}$ is a sequence of iid  geometrically distributed random variables such that 
\begin{equation*}
\mathrm{Pr}\{X_{1}=j\}=(1-\rho)\rho^{j-1},\ \ j\ge1,
\end{equation*}
where $0\leq\rho<1$. The sequence $\{X_{i}\}_{i\ge1}$ is independent of the Poisson process $\{N(t)\}_{t\ge0}$ whose intensity parameter is $\lambda$. Beghin and Macci (2014) introduced a fractional version of the PAP, namely, the fractional P\'olya-Aeppli process (FPAP) $\{\hat{N}_{\alpha}(t)\}_{t\ge0}$. It is defined as 
\begin{equation*}
\hat{N}_{\alpha}(t)\coloneqq \hat{N}(Y_{\alpha}(t)),\ \ 0<\alpha<1,
\end{equation*}
where $\{\hat{N}(t)\}_{t\ge0}$ and $\{Y_{\alpha}(t)\}_{t\ge0}$ are independent of each other.

On letting $k\to\infty$ in (\ref{cre}) with $\lambda_{j}=\lambda(1-\rho)\rho^{j-1}$ for all $j\ge1$ the system (\ref{cre}) reduces to the governing system of differential equations for the state probabilities of FPAP (see Beghin and Macci (2014), Eq. (19)).  Further, on taking $\alpha=1$, we get the governing system of differential equations for the state probabilities of PAP (see Chukova and Minkova (2013), Eq. (10)). Here, $\Lambda=\lambda$. Thus, the PAP and its fractional version FPAP are  obtained as a limiting process of GFCP. From Lemma \ref{lemma}, we get the following limiting result for the PAP:
\begin{equation*}
\lim_{t\to\infty}\frac{\hat{N}(t)}{t}=\frac{\lambda}{1-\rho},\ \ \text{in probability}.
\end{equation*}	

\subsection{Negative binomial process and its fractional version}
The negative binomial process  (NBP) $\{\bar{N}(t)\}_{t\ge0}$ is a compound Poisson process  studied by Kozubowski and Podg\'orski (2009). It is defined as  
\begin{equation*}
\bar{N}(t)\coloneqq\sum_{i=1}^{N(t)}X_{i},
\end{equation*}
where $\{X_{i}\}_{i\ge1}$ is a sequence of iid random variables with discrete logarithmic distribution such that    
\begin{equation*}
\mathrm{Pr}\{X_{1}=j\}=\frac{(1-p)^{j}}{j\ln(1/p)},\ \ j\ge1,
\end{equation*}
where $0<p<1$. The sequence $\{X_{i}\}_{i\ge1}$ is independent of the Poisson process $\{N(t)\}_{t\ge0}$ whose intensity parameter is $\ln(1/p)$. Beghin and Macci (2014) studied a fractional version of the NBP, namely, the fractional negative binomial process (FNBP) which we denote by $\{\bar{N}_{\alpha}(t)\}_{t\ge0}$. It is defined as 
\begin{equation*}
\bar{N}_{\alpha}(t)\coloneqq \bar{N}(Y_{\alpha}(t)),\ \ 0<\alpha<1,
\end{equation*}
where $\{\bar{N}(t)\}_{t\ge0}$ and $\{Y_{\alpha}(t)\}_{t\ge0}$ are independent of each other.

On letting $k\to\infty$ in (\ref{cre}) with $\lambda_{j}=(1-p)^j/j$ for all $j\ge1$ the system (\ref{cre}) reduces to the governing system of differential equations for the state probabilities of FNBP (see Beghin (2015), Eq. (66)). Here, $\Lambda=\ln(1/p)$. Thus, the FNBP is  obtained as a limiting process of GFCP. From Lemma \ref{lemma}, we get the following limiting result for the NBP:
\begin{equation*}
\lim_{t\to\infty}\frac{\bar{N}(t)}{t}=\frac{1-p}{p},\ \ \text{in probability}.
\end{equation*}	

\section{An application to risk theory}\label{section5}
Consider the following risk model with GCP as the counting process: 
\begin{equation}\label{risk}
X(t)=ct-\sum_{j=1}^{M(t)}Z_{j}, \ \ t\ge0,
\end{equation}
where $c>0$ denotes the constant premium rate. Here, $\{Z_{j}\}_{j\ge1}$ is the sequence of positive iid random variables with commmon distribution $F$. The $Z_j$'s represent the claim sizes and these are independent of the GCP.

Let $\mu=\mathbb{E}(Z_{j})$. The relative safety loading factor $\eta$ for the risk model (\ref{risk}) is given by 
\begin{equation*}
\eta=\frac{\mathbb{E}(X(t))}{\mathbb{E}\left(\sum_{j=1}^{M(t)}Z_{j}\right)}=\frac{c}{\mu\sum_{j=1}^{k}j\lambda_{j}}-1.
\end{equation*}
Hence, $c>\mu\sum_{j=1}^{k}j\lambda_{j}$ holds when the safety loading factor is positive. Let $u\geq0$ denote the initial capital and $\{U(t)\}_{t\ge0}$ be the surplus process where $U(t)=u+X(t)$.

Let $\tau$ denote the time to ruin of an insurance company. So, 
\begin{equation*}
\tau=\inf\{t>0:U(t)<0\},
\end{equation*}
 where $\inf\phi=\infty$ and the ruin probability is given by $\psi(u)=\mathrm{Pr}\{\tau<\infty\}$. Let $ G(u,y)$ be the joint probability that the ruin occurs in finite time and the deficit $D=|U(\tau)|$ at the time of ruin is not more than $y\ge0$, that is,
 \begin{equation}\label{wsqa11}
 G(u,y)=\mathrm{Pr}\{\tau<\infty, D\le y\}.
 \end{equation} 
Note that 
 \begin{equation*}
 \lim\limits_{y\to \infty}G(u,y)=\psi(u).
 \end{equation*}
 
The transition probabilities of GCP are given by (see Di Crescenzo {\it et al.} (2016), Section 2)
\begin{equation}\label{pure}
\mathrm{Pr}\{M(t+h)=n|M(t)=m\}=\begin{cases*}
1-\Lambda h+o(h),\quad n=m,\\
\lambda_{j}h+o(h), \quad n=m+j,\ \ j=1,2,\dots,k,\\
o(h),\ \ n>m+k,
\end{cases*}
\end{equation}
where $\Lambda=\lambda_{1}+\lambda_{2}+\dots+\lambda_{k}$. 

Let $\bar{u}=u+ch$ and $F^{*j}(\cdot)$ be the distribution of $Z_{1}+Z_{2}+\dots+Z_{j}$ for all $j=1,2,\dots,k$. Then, from (\ref{wsqa11}) and (\ref{pure}), we get
\begin{equation*}
G(u,y)=(1-\Lambda h)G(\bar{u},y)+o(h)+\sum_{j=1}^{k}\lambda_{j}h\left(F^{*j}(\bar{u}+y)-F^{*j}(\bar{u})+\int_{0}^{\bar{u}}G(\bar{u}-x,y)\mathrm{d}F^{*j}(x)\right)
\end{equation*}
which can be rewritten as follows:
\begin{align}
\frac{G(\bar{u},y)-G(u,y)}{ch}&=\frac{\Lambda}{c}G(\bar{u},y)+\frac{o(h)}{h}\nonumber\\
&\ \ -\frac{1}{c}\sum_{j=1}^{k}\lambda_{j}\left(F^{*j}(\bar{u}+y)-F^{*j}(\bar{u})+\int_{0}^{\bar{u}}G(\bar{u}-x,y)\mathrm{d}F^{*j}(x)\right).\label{gvfddds21}
\end{align}
Let 
\begin{equation*}
H(x)=\frac{1}{\Lambda}\sum_{j=1}^{k}\lambda_{j}F^{*j}(x)
\end{equation*}
be the mixture distribution whose mixture components are the distributions $F^{*j}(\cdot)$'s of the aggregated claims $Z_{1}+Z_{2}+\dots+Z_{j}$. It follows that $H(0)=0$, $H(\infty)=1$. On letting $h\to 0$ in (\ref{gvfddds21}), we get
\begin{equation}\label{(u,y)}
\frac{\partial G(u,y)}{\partial u}=\frac{\Lambda}{c}\left(G(u,y)+H(u)-H(u+y)-\int_{0}^{u}G(u-x,y)\mathrm{d}H(x)\right).
\end{equation}
Further, on letting $y\to\infty$, we get
\begin{equation*}
	\frac{\mathrm{d}}{\mathrm{d}u}\psi(u)=\frac{\Lambda}{c}\left(\psi(u)+H(u)-1-\int_{0}^{u}\psi(u-x)\mathrm{d}H(x)\right).
	\end{equation*}
\begin{theorem}
	The function $G(0,y)$ is given by
	\begin{equation}\label{G(0,y)}
G(0,y)=	\frac{\Lambda}{c}\int_{0}^{y}(1-H(u))\mathrm{d}u.
	\end{equation}
\end{theorem}
\begin{proof}
	Integrating (\ref{(u,y)}) with respect to $u$ from $0$ to $\infty$ and using $G(\infty,y)=0$, we get
	\begin{equation*}
	-G(0,y)=\frac{\Lambda}{c}\left(\int_{0}^{\infty}G(u,y)\mathrm{d}u+\int_{0}^{\infty}(H(u)-H(u+y))\mathrm{d}u-\int_{0}^{\infty}\int_{0}^{u}G(u-x,y)\mathrm{d}H(x)\mathrm{d}u\right).
	\end{equation*}
	Thus, the change of variable yields 
	\begin{align*}
	G(0,y)=\frac{\Lambda}{c}\int_{0}^{\infty}(H(u+y)-H(u))\mathrm{d}u=\frac{\Lambda}{c}\int_{0}^{y}(1-H(u))\mathrm{d}u.
	\end{align*}
	This completes the proof.
\end{proof}
\begin{corollary}
The ruin probability for zero initial capital is given by
\begin{equation*}
\psi(0)=\frac{\mu}{c}\sum_{j=1}^{k}j\lambda_{j}.
\end{equation*}
\end{corollary}
\begin{proof}
On taking limit $y\to \infty$ in (\ref{G(0,y)}), we get
	\begin{equation*}
	\psi(0)=\frac{\Lambda}{c}\int_{0}^{\infty}(1-H(u))\mathrm{d}u.
	\end{equation*}
	Let $X$ be a random variable with distribution function $H(x)$. Then,
	\begin{equation*}
	\mathbb{E}(X)=\frac{\mu}{\Lambda}\sum_{j=1}^{k}j\lambda_{j}.
	\end{equation*}
Using the fact that $\mathbb{E}(X)=\displaystyle\int_{0}^{\infty}(1-H(u))\mathrm{d}u$, we get the required result.
\end{proof}
\section{Concluding remarks}
We obtain some additional results and study new properties for the GFCP, a fractional counting process introduced by Di Crescenzo {\it et al.} (2016). Its $r$th factorial moment and the covariance are derived. We establish the LRD and SRD properties for it and its increments, respectively. It is shown that the GFCP is a scaling limit of some CTRW. A particular case of the GFCP, namely, the GCP is discussed in detail for which we obtain a martingale result and establish a recurrence relation for its pmf. We obtain a limiting result for the GCP using which we prove that the one-dimensional distributions of GFCP are not infinitely divisible. It is shown that many known counting processes recently introduced and studied by several researchers such as the PPoK, PAP, PAPoK, NBP and their fractional versions are special cases of the GFCP. An application of the GCP to ruin theory is discussed where it is used as a counting process.

\end{document}